\documentclass[reqno, 12pt]{amsart}

\usepackage[utf8]{inputenc}
\usepackage[T1]{fontenc}
\usepackage{lmodern}
\usepackage{amsmath}
\usepackage{amsthm}
\usepackage{amssymb}
\usepackage{enumerate}
\usepackage{mathrsfs}
\usepackage{graphicx}
\usepackage[all]{xy}
\usepackage{url}
\usepackage{hyperref}
\usepackage[a4paper, hmargin=3cm, vmargin=3cm]{geometry}

\newcommand{\ud}{\,\mathrm{d}}

\newcommand{\N}{\mathbb{N}}
\newcommand{\Z}{\mathbb{Z}}

\newcommand{\R}{\mathbb{R}}

\newcommand{\la}{\left\langle}
\newcommand{\ra}{\right\rangle}

\newcommand{\id}{\textup{id}}
\renewcommand\Im{\textup{Im}}

\newtheorem{thm}{Theorem}[section]
\newtheorem{lem}[thm]{Lemma}

\newtheorem{prop}[thm]{Proposition}

\newtheorem{prop-def}[thm]{Définition-proposition}

\theoremstyle{definition}

\theoremstyle{remark}

\title{On the growth rate of geodesic chords}

\author[S. Allais]{Simon Allais}
\address{Simon Allais,
\'Ecole Normale Sup\'erieure de Lyon,
UMPA\newline\indent  46 all\'ee d'Italie,
69364 Lyon Cedex 07, France}
\email{simon.allais@ens-lyon.fr}
\urladdr{http://perso.ens-lyon.fr/simon.allais/}
\date{November 28, 2019}

\begin{document}
\maketitle
\newcommand\heq{\approx}
\newcommand\sphere[1]{\mathbb{S}^{#1}}
\newcommand\ind{\textup{ind}}
\newcommand\length{L}
\newcommand\CL[2][p,q]{N( #2 ; #1 )}
\newcommand\CLd[2][p,q]{n( #2 ; #1 )}
\newcommand\CLdo[2][p,q]{n_+( #2 ; #1 )}

\begin{abstract}
    We show that every forward complete Finsler manifold 
    of infinite fundamental group and
    not homotopy-equivalent to $S^1$
    has infinitely many geometrically distinct geodesics joining 
    any given pair of points $p$ and $q$.
    In the special case in which $\beta_1(M;\Z)\geq 1$ and $M$ is closed,
    the number of geometrically
    distinct geodesics between two points grows at least logarithmically.
\end{abstract}

\section{Introduction}

Let $M$ be a forward complete Finsler manifold of infinite fundamental group
(every manifold $M$ will be assumed to be connected).
We are interested in the growth rate of geodesics
joining two arbitrarily given points
$p,q\in M$, and especially in asymptotic properties that only involve
the topology of $M$.
Two paths $\gamma:[0,1]\to M$ and $\delta:[0,1]\to M$
are said to be geometrically distinct
if their images are distinct.
For $\ell >0$, we denote by $\CLd{\ell}$ the number of geometrically distinct
geodesics between $p$ and $q$ of length $\leq\ell$.
It is well known that for $\pi_1 (M)$ ``large enough'' this number
tends to infinity without any further assumption.
A precise statement is the following:

\begin{prop}\label{prop:classicGrowth}
    Let $M$ be a manifold such that
    $\pi_1(M)$ has a polynomial growth of degree $d>1$.
    For each forward complete Finsler metric on $M$,
    there exist continuous functions $a:M\to (0,+\infty)$
    and $b:M\to\R$ such that
    \begin{equation*}
        \CLd{\ell}\geq a(q)\ell^{d-1} + b(q),
        \quad\forall p,q\in M.
    \end{equation*}
\end{prop}

For the reader's convenience, we add the proof of this result,
which is certainly well known to the experts.
We are interested in the remaining case in which the growth rate of $\pi_1(M)$ is linear.
In fact, we show the following general result:

\begin{thm}\label{thm:infinite}
    Let $M$ be a manifold 
    of infinite fundamental group $\pi_1(M)$
    and not homotopy-equivalent to $S^1$.
    Then, given any forward complete Finsler metric on $M$,
    \begin{equation*}
        \CLd{\ell}\to +\infty\quad
        \text{as}\quad \ell\to+\infty,\quad
        \forall p,q\in M.
    \end{equation*}
\end{thm}

Of course, the assertion of Theorem~\ref{thm:infinite} does not hold
for the flat cylinders $S^1\times\R^n$,
which are homotopy-equivalent to $S^1$.
In his Ph.D. thesis,
Mentges proved Theorem~\ref{thm:infinite} in the case 
where $p=q$ and the universal cover of $M$ is not contractible
\cite[{\emph{Satz}~2.2.1.}]{Men87}.
We can be more specific when $H_1(M;\Z)$ has non-zero rank:

\begin{thm}\label{thm:log}
    Let $M$ be a closed manifold not homotopy-equivalent
    to $S^1$ (that is any closed $M$ of dimension $\geq 2$) and
    with first Betti number $\beta_1(M;\Z)\geq 1$.
    Then, given any Finsler metric on $M$,
    there exist $a>0$ and $b\in\R$ such that
    \begin{equation*}
        \CLd{\ell}\geq a\log\ell + b,\quad
        \forall\ell>0,\forall p,q\in M.
    \end{equation*}
\end{thm}

\begin{thm}\label{thm:loglog}
    Let $M$ be a manifold not homotopy-equivalent
    to $S^1$ and
    with first Betti number $\beta_1(M;\Z)\geq 1$.
    Then, given any forward complete Finsler metric on $M$,
    there exists a continuous function
    $b:M\to\R$ such that
    \begin{equation*}
        \CLd{\ell}\geq \frac{\log(\log\ell)}{2\log 2} + b(q),\quad
        \forall\ell>0,\forall p,q\in M.
    \end{equation*}
\end{thm}

When the universal cover of $M$ is not contractible
(that is $M$ is not an Eilenberg-MacLane space),
Theorems~\ref{thm:infinite}, \ref{thm:log} and \ref{thm:loglog}
are deduced from a min-max argument
inspired by Bangert-Hingston \cite{BH}.
When $M$ has a contractible universal cover,
the estimate is even stronger, since the growth is at least linear:

\begin{lem}\label{lem:EMcL}
    Let $M$ be a manifold not homotopy-equivalent to $S^1$
    and with a contractible universal cover.
    Then, $\pi_1(M)$ has at least a quadratic growth rate.
\end{lem}

We notice that any closed manifold of dimension $\geq 2$ with 
a contractible universal cover satisfies the above condition.

Investigations on the links between the number of geodesics joining two points
and the manifold topology 
go back to Morse seminal works \cite{Mor32, Mor37}, where
he proved that any couple of points of a closed Riemannian manifold $M$ can be joined
by infinitely many geodesics provided that the homology groups
of the loop space of $M$ have a non trivial rank in infinitely
many degrees.
Serre proved that this assumption is always satisfied for simply connected $M$
by studying the spectral sequence associated to
the fibration $\textup{ev}:P\to M$, where $P$ is the space of
paths $\gamma\in C^0([0,1],M)$ such that $\gamma(0)=p$ is
a fixed base point and $\textup{ev}(\gamma):=\gamma(1)$
\cite[Prop. IV.11]{Ser51}.
In the above result, geodesics are not necessarily geometrically
distinct.
Proving that there are infinitely many geometrically distinct geodesics
with Morse theory is more subtle.
Inspired by Gromoll-Meyer (see below), Tanaka
\cite[Problem~C]{Tan82} asked if it is enough for
a simply connected Riemannian manifold $(M,g)$ to assume
that the sequence of Betti numbers of the loop space $(\beta_i(\Omega M))$
on some field is unbounded to get that $\CLd{\ell}\to\infty$
for any pair of points $p,q\in M$.
When $p$ and $q$ are non-conjugate, he sketched the proof
and Caponio-Javaloyes \cite{CJ13} gave a detailed proof
in the more general case of a connected, forward and backward complete
Finsler manifold.
When $p$ and $q$ are conjugate, it is still an open problem.

For the related problem of closed geodesics in Riemannian manifold,
one of the first results in that direction was due to Gromoll-Meyer \cite{GM69a}:
if the sequence of the Betti numbers of the free loop space
of a simply connected  Riemannian manifold $M$ is unbounded, then
there are infinitely many geometrically distinct closed geodesics
on $M$.
As for the growth rate, Bangert and Hingston proved that it is
at least like the one of prime numbers (\emph{i.e.}  $\gtrsim \frac{\ell}{\log\ell}$)
up to a multiplicative and an additive constant
when $\pi_1(M)$ is infinite \emph{abelian} \cite{BH},
or when $M=\sphere{2}$ \cite{Hin93}.
The obstruction for us to find such a better growth
seems to come from the lack of iteration map
$\gamma\mapsto(t\mapsto\gamma(mt))$, $m\in\N^*$, in the space
of geodesic chords joining $p$ and $q$.
This result was extended by Ta\u{\i}manov
to a large class of infinite non-abelian fundamental groups in \cite{Tai93}.
Nevertheless, the existence of infinitely many geometrically distinct closed geodesics
in closed Riemannian manifolds with a general infinite fundamental group
is still an open problem.

\subsection*{Organisation of the paper}
In Section~\ref{se:prelim}, we fix the notation and the conventions
on the objects that we will use throughout the paper, and
we briefly recall the variational theory of geodesics for a Finsler manifold.
In Section~\ref{se:growth}, we give a proof of Proposition~\ref{prop:classicGrowth}
and Lemma~\ref{lem:EMcL}.
In Section~\ref{se:noEMcL},
we prove Theorems~\ref{thm:infinite}, \ref{thm:log} and \ref{thm:loglog}.

\subsection*{Acknowledgments}
I am very grateful to my advisor Marco Mazzucchelli
who introduced me to min-max techniques and their
relation to geodesics.

\section{Preliminaries}\label{se:prelim}

\subsection{Definitions and conventions on path spaces}\label{se:def}

Let $M$ be a connected manifold (every manifold will be assumed to be connected).
We fix, once for all, an auxiliary complete Riemannian metric $g_0$ on $M$.
By $H^1$-path, we mean an absolutely continuous function
$\gamma:[0,1]\to M$ such that the integral
$\int_0^1 g_0(\gamma',\gamma')\ud t$ is finite.
For $p,q\in M$ let $\Omega_{p,q}$ be the set of $H^1$-paths
$\gamma : [0,1]\to M$ with end-points $\gamma(0)=p$ and $\gamma(1)=q$
and $\Omega_p := \Omega_{p,p}$.
For $\gamma,\delta\in\Omega_{p,q}$, we write $\gamma\heq\delta$
if $\gamma$ and $\delta$ belong to the same path-connected component of $\Omega_{p,q}$.
For $\gamma\in\Omega_{p,q}$ and $\delta\in\Omega_{q,r}$,
we denote by $\gamma\cdot\delta\in\Omega_{p,r}$ the chained path
$t\mapsto \gamma(2t)$ for $t\in[0,1/2]$ and $t\mapsto\delta(2t-1)$
for $t\in[1/2,1]$.
We denote $a\cdot b\cdot c = (a\cdot b)\cdot c$ so that
$a\cdot b\cdot c\heq a\cdot(b\cdot c)$.
For $\gamma\in\Omega_{p,q}$, let $\gamma^{-1}\in\Omega_{q,p}$ be 
the reversed path $t\mapsto\gamma(1-t)$, so that
$\gamma\cdot\gamma^{-1}\heq \bar{p}$ where $\bar{p}\in\Omega_{p}$
denotes the constant path.
If $\gamma\in\Omega_q$ for some $q\in M$, $[\gamma]_{\pi_1}\in\pi_1(M,q)$
denotes its class in the fundamental group,
or simply $[\gamma]$ if there is no ambiguity on notation.

Let $p,q,p',q'\in M$, $\alpha\in\Omega_{p,p'}$ and
$\beta\in\Omega_{q,q'}$, then
$f:\Omega_{p,q}\to\Omega_{p',q'}$ and $g:\Omega_{p',q'}\to\Omega_{p,q}$
defined by $f(\gamma)=\alpha^{-1}\cdot\gamma\cdot\beta$
and $g(\gamma)=\alpha\cdot\gamma\cdot\beta^{-1}$ are
homotopy inverses, thus $\Omega_{p,q}$ and $\Omega_{p',q'}$ are
homotopy-equivalent spaces.
For all $h\in\pi_1(M,q)$ let $\Omega^h_q$ be the path-connected component such that
for all $\gamma\in\Omega^h_q$, $[\gamma]=h$.
We fix an arbitrary $\alpha\in\Omega_{p,q}$ once for all and we define
$\Omega^h_{p,q}:=\{\gamma\in\Omega_{p,q}\ |\ [\alpha^{-1}\cdot\gamma]=h\}$.

\subsection{Background on Finsler geodesics}\label{se:Finsler}

Let us recall some basic notion from Finsler geometry.
For a general reference, see \cite{BCS}.

Let $M$ be a manifold, $TM$ be its tangent bundle
and $\pi:TM\to M$ be the base projection.
A continuous function $F:TM\to [0,+\infty)$
is a Finsler metric if
\begin{itemize}
    \item it is smooth on $TM\setminus 0$,
        where $0\subset TM$ denotes the $0$-section,
    \item it is fiberwise positively homogeneous of degree 1,
        \emph{i.e.} $F(\lambda v)=\lambda F(v)$
        for $v\in TM$ and $\lambda >0$,
    \item its square $F^2$ is fiberwise strongly convex,
        that is the fundamental tensor
        \begin{equation*}
            g_u(v,w) := \left.\frac{1}{2}
            \frac{\partial^2}{\partial t\partial s}F^2(
            u+tv+sw)\right|_{t=s=0},\quad
            \forall v,w\in T_{\pi(u)}M,
        \end{equation*}
        is positive definite for every $u\in TM\setminus 0$.
\end{itemize}
A Finsler metric $F$ on $M$ induces a length on
$\Omega_{p,q}$, given by
\begin{equation*}
    \length(\gamma) := \int_0^1 F(\gamma'(t))\ud t,\quad
    \forall \gamma\in\Omega_{p,q}.
\end{equation*}
and a (not necessarily symmetric) distance $d$ on $M$, given by:
\begin{equation*}
    d(p,q) := \inf_{\gamma\in\Omega_{p,q}} \length(\gamma),\quad
    \forall p,q\in M.
\end{equation*}
Since $F(-v) = F(v)$ does not necessarily hold,
$d$ is not necessarily symmetric.
A sequence $(x_i)$ in $M$ is called a forward Cauchy sequence
if, for all $\varepsilon>0$, there exists a positive integer $N$
such that $N\leq i < j$ implies $d(x_i,x_j)<\varepsilon$.
The Finsler manifold $(M,F)$ is said to be \emph{forward complete}
if every forward Cauchy sequence converges in $M$.

Similarly to the Riemannian case, where $F$ is simply the
associated Riemannian norm, \emph{geodesics}
are the curves whose small portions are length minimizing. 
Moreover, they satisfy a
differential equation inducing an exponential map
between a neighborhood of $p\in M$ and a neighborhood of
$T_p M$.
If $F$ is forward complete, the Hopf-Rinow Theorem from Riemannian geometry
remains true in the Finsler setting:
the exponential map is onto
and, for all $p,q\in M$, there exists a geodesic in
$\Omega_{p,q}$ minimizing the length.

Throughout the paper, all the geodesics $\gamma$ will be considered parametrized
with constant speed equal to $F(\gamma')$
and we will often identify geo\-de\-sics and reparametrized geodesics
when writing ``$\delta\cdot\gamma$ is a geodesic''.
For $p,q\in M$, geodesics in $\Omega_{p,q}$ are then exactly the critical points
of the energy functional:
\begin{equation*}
    E(\gamma) := \int_0^1 F^2(\gamma'(t))\ud t,\quad
    \forall \gamma\in\Omega_{p,q}.
\end{equation*}
If the Finsler metric is forward complete, then $E:\Omega_{p,q}\to[0,+\infty)$ 
satisfies the Palais-Smale condition (see for example \cite[Section~3]{CJM}).
Given $\gamma\in\Omega_{p,q}$ critical, we denote by
$\ind(\gamma)$ its index,
which is the non-negative integer computed in the same way as in the Riemannian case
using Jacobi fields.
The index of a geodesic chord shares properties similar to the Riemannian case.
In particular, for two geodesic chords $\gamma\in\Omega_{p,q}$ and 
$\delta\in\Omega_{q,r}$:
\begin{enumerate}[(i)]
    \item \label{it:nullind} if $\ind(\gamma)=0$,
        then for $s\in(0,1)$, $\gamma|_{[0,s]}$
        reparametrized by constant speed on $[0,1]$
        is a local minimum of $E:\Omega_{p,\gamma(s)}\to\R$,
    \item \label{it:addind} If $\gamma\cdot\delta$ is a geodesic, then
        $\ind(\gamma\cdot\delta)\geq \ind(\gamma)+\ind(\delta)$.
\end{enumerate}
However, $E$ is only of class $C^{1,1}$ in general and 
it can be very technical to make this functional
fit into the Morse \emph{apparatus}.
To overcome this issue, we can retract $\{ E < \lambda \}$
to a finite dimensional subspace $B$ of broken geodesics joining $p$ and $q$.
We briefly recall the construction of $B\subset\{ E<\lambda\}$
and its retraction $(r_s)$
(a comprehensive reference for the Riemannian case
is \cite[Part~III, \S16]{Mil63}).
Let $k\geq 1$ be large so that for all $\gamma\in\{ E<\lambda\}$
there exists a unique minimizing geodesic joining
$\gamma(s)$ and $\gamma(t)$ for $|s-t|\leq 1/k$.
Let 
\begin{equation*}
B :=\left\{ \gamma\in\{ E <\lambda\}\ |\ 
\gamma|_{[i/k,(i+1)/k]} \text{ is a geodesic}, 0\leq i\leq k-1 \right\}
\end{equation*}
be the subspace of $k$-broken geodesics $\subset\{ E<\lambda\}$.
It is a finite dimensional manifold since it is diffeomorphic
to an open subset of the $(k-1)$-fold product $M\times\cdots\times M$
\emph{via} $\gamma\mapsto (\gamma(1/k),\ldots,\gamma((1-k)/k))$.
The retraction homotopy $r_s : \{ E <\lambda \}\to
\{ E <\lambda \}$, 
with $r_0 =\id$ and $\Im(r_1)=B$, is defined as follows.
For each $\gamma\in\{ E<\lambda\}$, 
$r_{s}(\gamma)$ coincides with $\gamma$ everywhere except
on intervals of the form $[i/k,(i+s)/k]$, $0\leq i\leq k-1$,
and the restrictions of $r_s(\gamma)$ to such intervals are
minimizing geodesics.
This retraction has the following properties:
\begin{enumerate}[(a)]
    \item\label{item:dec}
        $\forall s\in[0,1]$, $E\circ r_s \leq E$,
    \item if $\gamma\in\{ E<\lambda\}$ is a geodesic,
        $r_s (\gamma) \equiv \gamma$,
    \item\label{item:crit}
        critical points of $E|_B$ are exactly the critical points
        of $E|_{\{E<\lambda\}}$,
        $E|_B$ is smooth in their neighborhood
        and their Morse index is equal to their index defined
        with Jacobi fields.
\end{enumerate}

\section{growth rate of geodesic chords and growth rate of $\pi_1(M)$}\label{se:growth}

Throughout this section, $M$ is a forward complete Finsler manifold.

\subsection{Growth rate of geodesic chords and growth of the fundamental group}

We suppose that $\pi_1(M,q)$ is a finitely generated group
and denote by $e\in\pi_1(M,q)$ its neutral element.
Let $S\subset\pi_1(M,q)$ be a finite set of generators.
We recall that the word length of $g\in\pi_1(M,q)$ associated to $S$ is
\begin{equation*}
    |g| := \min\left\{ m\in\N\ |\
    \exists g_1,\ldots, g_m\in S\cup S^{-1}\cup\{e\},
    \ g=g_1\cdots g_m\right\}\in\N
\end{equation*}
and we denote the associated ball of radius $r\in\N$ by $B_r := \{g\in\pi_1\ |\ |g|\leq r\}$.
We will say that $\pi_1(M,q)$ has at least a polynomial growth rate of degree $d>0$
if there exists some $a>0$ such that $\# B_r \geq ar^d$.
This notion is indeed independent of the choice of $S$.

For $h\in\pi_1(M,q)$, we fix an arbitrary $\gamma_h\in\Omega^h_{p,q}$ minimizing the length,
it gives us a family of homotopically (but not geometrically) distinct
geodesics $(\gamma_h)_{h\in\pi_1}$
(where $\pi_1 = \pi_1(M,q)$ by a slight abuse of notation).

\begin{proof}[Proof of Proposition~\ref{prop:classicGrowth}]
    We suppose $\pi_1(M,q)$ has at least a polynomial growth rate of degree $d>0$.
    We take a finite generating part $S:=\{s_1,\ldots,s_n\}$, which is
    symmetric: $S=S^{-1}$ and contains the neutral element,
    and define the balls $B_r\subset X_S$ as above. 
    Let $c_1,\ldots,c_n\in\Omega_q$ be such that $[c_i]=s_i$
    and $c_i$ is minimizing length in its homotopy class.
    We will first give a lower bound on the counting number
    $\CL{\ell}$ of geodesics between $p$ and $q$ not necessarily geometrically
    distinct.

    Given $r\in\N$, let $g\in B_r$. There exist $i_1,\ldots,i_r\in\{ 1,\ldots, n\}$
    such that $g=s_{i_1}\cdots s_{i_r}$, so that
    $[\alpha^{-1}\cdot\gamma_g]=[c_{i_1}\cdots c_{i_r}]$
    (recall that $\alpha\in\Omega_{p,q}$).
    Since $\gamma_g$ is minimizing length in its homotopy class,
    \begin{equation*}
        \length(\gamma_g)\leq \length(\alpha\cdot c_{i_1}\cdots c_{i_r})
        \leq \length(\alpha) + r\max(\length(c_j)).
    \end{equation*}
    Therefore, since $(\gamma_g)_{g\in B_r}$ is a family of distinct
    geodesics, there exists $a>0$ depending only on the growth rate of $\pi_1(M,q)$
    such that
    \begin{equation*}
         \CL{\ell}\geq 
        a\left(\frac{\ell -\length(\alpha)}{\max(\length(c_i))}\right)^d,
        \quad\forall\ell >0.
    \end{equation*}

    We remark that there exists some positive number $b(p)>0$ 
    depending only on the Finsler metric on $M$ such that
    any $k$-iterate closed geodesic containing $p$ has length $\geq b(p)k$
    (one can take $b(p)$ to be twice the injectivity radius at $p$).
    Since a geodesic in $\Omega_{p,q}$ whose image appears
    multiple times can be uniquely written $d\cdot c^k$ 
    with a primitive closed geodesic $c\in\Omega_q$
    and a specific choice of geodesic chord $d\in\Omega_{p,q}$
    (look at the definition of a primitive geodesic chord
    in Section~\ref{se:minmax} for the precise statement),
    \begin{equation*}
         \CLd{\ell}\geq
        \frac{b}{2\ell}\CL{\ell} \geq a'\ell^{d-1} + b',\quad
        \forall\ell >0,
    \end{equation*}
    where $a'>0$ and $b'\in\R$ depend only on the metric and
    the growth rate of $\pi_1(M)$ and can be made continuous in $p\in M$.
\end{proof}

\subsection{Growth rate of the fundamental group of $K(\pi_1,1)$ closed manifolds}

Let $M$ be a $K(\pi_1,1)$ manifold with an infinite fundamental group
and a contractible universal cover.
According to Smith's theorem (see for instance
\cite[Theorem~16.1, page 287]{Hu59}), $\pi_1(M)$ is torsion-free.
We suppose that every finitely generated subgroup of 
$\pi_1(M)$ grows strictly less than any quadratic polynomial.
Then according to a deep result of Gromov \cite{Gro81}, $\pi_1(M)$ must
be virtually isomorphic to $\Z$:
that is $\pi_1(M)$ has a subgroup of finite index which is
isomorphic to $\Z$.
Alternatively, the reader can find 
an elementary proof of this statement in the
finitely generated case in \cite{WvdD}.
For a precise proof of the case where $\pi_1(M)$ is \emph{a priori}
infinitely generated, see \cite[Theorem~2]{Ma07}
(with notations of this theorem, since $G$ is torsion-free,
$L$ is trivial and $G$ is virtually $N\simeq\Z$).

The following algebraic lemma is certainly well known,
but we add its proof here for the reader's convenience,
as we could not find it in the literature.

\begin{lem}\label{lem:virtualZ}
    Let $G$ be a torsion-free group,
    if $G$ is virtually isomorphic to $\Z$
    (\emph{i.e.} there exists a subgroup of finite index
    $H<G$ which is isomorphic to $\Z$)
    then $G\simeq \Z$.
\end{lem}

\begin{proof}
    Since $G$ is virtually isomorphic to a finitely generated group,
    $G$ is a finitely generated group.
    Let $S=\{s_1,\ldots,s_n\}$ be a set of generators of $G$
    and let $H<G$ be a subgroup of finite index isomorphic to $\Z$.
    If $C(s):=\{ g\in G\ |\ gs = sg\}$ denotes the centralizer of
    $s\in G$, then $C(s)\cap H$ is not trivial.
    Indeed, $C(s)$ is infinite (it contains $\la s\ra$ which
    is infinite by hypothesis for $s\neq e$ and $C(e)=G$)
    and there exists a finite sequence $(g_i)$ in $G$
    such that $G=\bigcup_{i}g_i H$, thus
    there exists some $g=g_i$ such that $C(s)\cap gH$ is infinite.
    Let $c,c'\in C(s)\cap gH$ be distinct. There are $h,h'\in H$
    such that $c=gh$ and $c'=gh'$, then
    $c^{-1}c' = h^{-1}h' \neq 1$ is in $C(s)\cap H$.

    Since a finite intersection of non-trivial subgroups of $H\simeq\Z$ has
    a finite index, $\bigcap_{i}C(s_i)\cap H$ has a finite index in $H$.
    Thus the centralizer $Z=\bigcap_{i}C(s_i)$ of $G$ has a
    finite index in $G$.
    According to a theorem of Schur, 
    it implies that the commutator subgroup $D:=[G,G]$
    is finite \cite[Theorem~5.32]{Rot95}.
    Since $G$ is torsion-free, $D$ is trivial and $G$ is abelian.
    An abelian torsion-free finitely generated group is isomorphic
    to $\Z^r$ and $\Z$ is the only one virtually isomorphic to $\Z$.
\end{proof}

\begin{proof}[Proof of Lemma~\ref{lem:EMcL}]
    Let $M$ be a $K(\pi_1,1)$ manifold.
    We assume that $\pi_1(M)$ grows less than a quadratic polynomial,
    so that it is virtually isomorphic to $\Z$.
    Since $\pi_1(M)$ is torsion-free,
    Lemma~\ref{lem:virtualZ} implies that $\pi_1(M)\simeq\Z$.
    Since $M$ is a $K(\Z,1)$ manifold, it is homotopy-equivalent
    to $S^1$.
\end{proof}

\section{Growth rate when the universal cover is not contractible}\label{se:noEMcL}

Throughout this section, $M$ is a forward complete Finsler manifold
with an infinite fundamental group and a non-contractible universal cover.

\subsection{A sequence of min-max geodesics}\label{se:minmax}

Since the universal cover of $M$ is not contractible, there is some $n>1$ such
that $\pi_n(M,q)\neq 0$, we fix such an $n>1$.
For all $h\in\pi_1(M,q)$, we recall that
$\gamma_h\in\Omega^h_{p,q}$ denotes a minimizing geodesic.
We take a non-zero class $\nu\in\pi_{n-1} (\Omega_q,\bar{q})\simeq\pi_n (M,q)$ and,
for all $h\in\pi_1(M,q)$,
let $\nu_h:=(\gamma_h)_* \nu$ be the induced non-zero class
of $\pi_{n-1}(\Omega^h_{p,q},\gamma_h)$.

To be more precise on the definition of $\nu_h$,
    let $x_0\in\sphere{n-1}$ be the base point and
    $d$ be the round distance on $\sphere{n-1}$.
    Let $f:(\sphere{n-1},x_0)\to (\Omega_q,\bar{q})$
    be a smooth function in the class $\nu$
    such that $f(s)=\bar{q}$ for all $s\in\sphere{n-1}$
    such that $d(x_0,s)<1$.
    For $h\in\pi_1(M,q)$,
    let $f_h:(\sphere{n-1},x_0)\to (\Omega^h_{p,q},\gamma_h)$ be essentially defined by
    $f_h(s):=\gamma_h\cdot f(s)$ for all $s\in\sphere{n-1}$.
    To be more specific, in order to have $f_h(x_0)=\gamma_h$,
    let $f_h(s) := \gamma_h\cdot f(s)$ for all $s\in\sphere{n-1}$
    such that $d(x_0,s)\geq 1$ and
    \begin{equation}\label{eq:fh}
        \forall t\in [0,1],\quad
        f_h(s)(t) :=
        \left\{\begin{array}{l l}
                \gamma_h(t/\lambda(s)),&
                    \textup{ if } t\in [0,\lambda(s)],\\
                q, &
                \textup{ otherwise},
        \end{array}\right.
    \end{equation}
    for all $s\in\sphere{n-1}$ such that $d(x_0,s)<1$,
    with $\lambda(s) = 1-\frac{d(x_0,s)}{2}\in [1/2,1]$.
    Then we define $\nu_h$ as the homotopy class of $f_h$.

We suppose that $E:\Omega_{p,q}\to\R$ has a discrete set of critical points
(otherwise the conclusions of Theorems~\ref{thm:infinite},
\ref{thm:log} and \ref{thm:loglog} are clearly true)
and consider the min-max:
\begin{equation}\label{eq:minmax}
    \tau_h = \inf_{f\in\nu_h} \max_{s\in\sphere{n-1}} E(f(s)).
\end{equation}
Then $\tau_h$ is a critical value of $E$
and there exists a critical point $\delta_h\in\Omega_{p,q}$
of value $\tau_h$ which is not a local minimum
and satisfies $\ind(\delta_h)\leq n-1$.
This is a classical result if $E$ is $C^2$
in the neighborhood of its critical points and satisfies Palais-Smale
(see for instance \cite[Chapter~II]{Chang}).
Even though $E$ is not $C^2$ in any neighborhood of its critical points
for a general Finsler metric, we can apply a retraction
$(r_s)$ of $\{ E<\lambda\}$ for some $\lambda > \tau_h$
to a finite dimensional subspace $B$ of broken geodesics,
as explained above.
We thus see that $\tau_h$ satisfies (\ref{eq:minmax})
restricted to the finite dimensional subspace $B$,
according to property (\ref{item:dec}) of $(r_s)$.
Now property (\ref{item:crit}) allows us to find
$\delta_h\in\Omega_{p,q}$ among the critical points of $E|_B$
which is $C^2$ in the neighborhood of its critical points.

The following estimate will be useful in Section~\ref{se:betti}:

\begin{lem}\label{lem:tau}
    There exists a constant $C>0$ such that, for all $h\in\pi_1(M,q)$,
    \begin{equation*}
        \length(\gamma_h)^2 \leq \tau_h \leq 2\length(\gamma_h)^2 + C.
    \end{equation*}
\end{lem}

\begin{proof}
    Given $f_h\in\nu_h$, for all $s\in\sphere{n-1}$,
    $f_h(s)\heq \gamma_h$ and $\gamma_h$ is minimizing length with a constant velocity,
    thus $E(\gamma_h)\leq E(f_h(s))$ which gives $\length(\gamma_h)^2\leq \tau_h$.

    Let $x_0\in\sphere{n-1}$ be the base point and
    $d$ be the round distance on $\sphere{n-1}$.
    Let $f:(\sphere{n-1},x_0)\to (\Omega_q,\bar{q})$
    be a smooth function in the class $\nu$
    such that $f(s)=\bar{q}$ for all $s\in\sphere{n-1}$
    such that $d(x_0,s)<1$
    and define $f_h\in\nu_h$ by (\ref{eq:fh}).
    Then, for all $s\in\sphere{n-1}$,
    \begin{equation*}
        E(f_h(s))\leq  2E(\gamma_h) + 2E(f(s)),
    \end{equation*}
    thus $\tau_h\leq 2E(\gamma_h) + C$ with $C:=2\max E\circ f$ independent of $h$.
\end{proof}

We recall that a geodesic loop $c\in\Omega_q$ is \emph{primitive}
if there does not exist any geodesic loop $c_0\in\Omega_q$
and any positive integer $k>1$ such that $c=c_0^k$.
We will say that a geodesic chord $d\in\Omega_{p,q}$
is \emph{primitive} if 
there does not exist any 
geodesic loop $c\in\Omega_q$ such that $\Im(d)=\Im(c)$
or if it is a primitive geodesic loop
(which is only possible in the case $p=q$).
For all geodesic chord $\beta\in\Omega_{p,q}$,
there exists a unique primitive geodesic chord $d\in\Omega_{p,q}$
such that $\beta = d\cdot c^k$,
where $c\in\Omega_q$ is either $\bar{q}$ or
the primitive geodesic loop containing $d$
and $k\in\N$.
We will say that the geodesic chord $\beta\in\Omega_{p,q}$
\emph{carries} the primitive chord $d\in\Omega_{p,q}$.
Thus if the family $(\delta_h)$ carry $m$ distinct primitive
chords $d_1,\ldots,d_m$ and $p\neq q$,
then $d_1,\ldots,d_m$ are geometrically distinct
geodesic chords joining $p$ and $q$.
In the special case $p=q$, 
it is possible that $d_r = d_s^{-1}$ for some $r\neq s$
so that at least $\lceil m/2\rceil$ of them are geometrically distinct.

We now study the number of times chords carrying 
the same primitive chord $d\in\Omega_{p,q}$
can appear in the infinite family $(\delta_h)_{h\in\pi_1}$.
Let $d\in\Omega_{p,q}$ be a primitive geodesic chord.
Let $(\delta'_i)_{1\leq i\leq N}:=(\delta_{h_i})$, $N\geq 2$, be a sequence included in
$(\delta_h)$, $h\in\pi_1$,
such that each $\delta'_i$ carries $d$.
Since $N\geq 2$, the primitive chord $d$ is included
in some primitive geodesic loop $c\in\Omega_q$ and
there exists a sequence of non-negative integers $(k_i)_{1\leq i\leq N}$
such that $\delta'_i = d\cdot c^{k_i}$.
Since each $\delta'_i$ belongs to a different path-connected
component $\Omega^{h_i}_{p,q}\subset\Omega_{p,q}$,
the $k_i$'s are distinct.

Now we remark that $\ind (\delta'_i) \geq 1$ for $i=1$ or $i=2$:
otherwise if one supposes $k_2>k_1$ then
$\delta'_1$ must be a local minimum of $E$,
according to property (\ref{it:nullind}) of Section~\ref{se:Finsler}.
If one supposes $k_2>k_1$,
then $\ind (d\cdot c^{k_2})\geq 1$ implies that
$\ind (d\cdot c^{n(k_2 +1)})\geq n$
(property (\ref{it:addind}) of Section~\ref{se:Finsler}) thus
$k_i < n(k_2 +1)$ for all $i$.

\begin{proof}[Proof of Theorem~\ref{thm:infinite}]
    The case of a contractible universal cover is a consequence
    of Lemma~\ref{lem:EMcL}.

    In our present setting, the theorem follows from the fact that a same
    primitive chord joining $p$ and $q$ can only be carried 
    a finite number of time
    in the infinite family $(\delta_h)_{h\in\pi_1}$.
    Indeed, let $(\delta'_i)_{1\leq i\leq N}$ be
    a sequence inside the family with $N\geq 2$ possibly infinite,
    each $\delta'_i$ carrying the same primitive chord
    and let $(k_i)$ be the associated injective sequence in $\N$.
    Since the $k_i$'s are distinct, $N\leq n(\max(k_1,k_2) +1)$.
\end{proof}

\subsection{Logarithmic growth when $\beta_1(M;\Z)\geq 1$}\label{se:betti}

Let $M$ be a forward complete Finsler manifold with first Betti number
$\beta_1(M;\Z)\geq 1$ and which is not $K(\pi_1,1)$.

Let $h\in\pi_1(M,q)$ be such that its image under the Hurewicz map
$\pi_1(M,q)\to H_1(M;\Z)$
is of infinite order (in particular the order of $h$ is also infinite).
Here, for $m\in\Z$, 
$\Omega^m_{p,q} := \{ \gamma\in\Omega_{p,q}\ |\ [\alpha^{-1}\cdot\gamma]_{\pi_1}=h^m\}$,
where $\alpha\in\Omega_{p,q}$ is fixed once for all.
Let $\gamma_m:=\gamma_{h^m}$ be a global minimizer of $E$ on $\Omega^m_{p,q}$
and $\delta_m := \delta_{h^m}$.

\begin{lem}\label{lem:lbound}
    If $M$ is closed,
    there are $a,a'>0$ and $b,b'\in\R$ such that
    \begin{equation*}
        \forall m\in\N,\quad
        am+b\leq \length(\delta_m) \leq a'm+b'.
    \end{equation*}
    The inequality $\length(\delta_m)\leq a'm+b'$ still holds 
    when $M$ is not closed.
\end{lem}

\begin{proof}
    According to Lemma~\ref{lem:tau},
    it suffices to prove these inequalities for $\length(\gamma_m)$
    instead of $\length(\delta_m) = \sqrt{\tau_{h^m}}$.

    Let $c := \alpha^{-1}\cdot\gamma_1\in\Omega_q$,
    then $[c^m]= [\alpha^{-1}\cdot\gamma_1]^m = h^m = [\alpha^{-1}\cdot \gamma_m]$
    so $\alpha\cdot c^m \in\Omega^m_{p,q}$ and, since $\gamma_m$
    is minimizing the length,
    $\length(\gamma_m)\leq \length(c^m)+\length(\alpha)
    = m\length(c)+\length(\alpha)$.

    For the lower bound, it comes from the fact
    that $\length(\alpha^{-1}\cdot\gamma_m)\geq m\|[h]\|_{s}$,
    where $\| [h]\|_{s} >0$ is the stable norm of
    $[h]\in H_1(M,\R)$.
    To show that directly, one can take a 1-form $\omega$ such that
    $\la\omega,[h]\ra\neq 0$ (it exists since $[h]\neq 0$
    on $H_1(M,\R)$ by hypothesis) and remark that
    \begin{equation*}
        m|\la\omega,[h]\ra| =
        |\la\omega,[h^m]\ra| = \left| \int_{\alpha^{-1}\cdot\gamma_m}\omega\right|
        \leq \left(\sup_{x\in M} \|\omega_x\|\right) \length(\alpha^{-1}\cdot\gamma_m).
    \end{equation*}
    Since $M$ is closed, $\sup_{x\in M} \|\omega_x\|$ is finite.
\end{proof}

Thanks to the estimate on $\length(\delta_m)$ given by Lemma~\ref{lem:lbound},
we can be more specific on the number of times a same primitive chord
can be carried in $(\delta_m)$:

\begin{lem}\label{lem:linear}
    If $M$ is closed,
    there exist some $a,b>0$ such that
    if $\delta_{m_1}$ and $\delta_{m_2}$ 
    carry the same primitive chord $d$ for some $m_1<m_2$,
    then for all $m>am_2 +b$ the chord
    $\delta_m$ does not carry $d$.
\end{lem}

\begin{proof}
    Let $(\delta'_i) = (d\cdot c^{k_i})_i$ be the sub-sequence
    $(\delta_{m_i})$ of $(\delta_m)$ carrying the primitive chord $d$
    with $(m_i)$ being increasing
    and let $\kappa :=\max(k_1,k_2)\geq 1$.
    We have seen that $\ind(d\cdot c^\kappa)\geq 1$,
    thus $\ind(d\cdot c^{n(\kappa +1)})\geq n$ which
    implies that the finite sequence $(k_i)$ is bounded by
    $n(\kappa +1)$.
    According to Lemma~\ref{lem:lbound}, for $i=1$ and $i=2$,
    $\length(\delta'_i) = \length(d) +k_i \length(c) \leq a'm_i +b'
    \leq a'm_2+b'$ so $\kappa \leq \frac{a'm_2 +b'}{\length(c)}$.

    Then the lower bound of Lemma~\ref{lem:lbound} together
    with $k_i\leq n(\kappa +1)$ implies that
    \begin{equation*}
        m_i \leq \frac{n}{a}\kappa\length(c) +  \frac{n\length(c) +\length(d) - b}{a}.
    \end{equation*}
    Since $\kappa$ is non-zero,
    $\length(c)+\length(d)\leq a' m_2 + b'$
    and finally $m_i \leq 2\frac{a'}{a}nm_2 + 2n\frac{b'}{a}-\frac{b}{a}$.
\end{proof}

\begin{proof}[Proof of Theorem~\ref{thm:log}]
Let $A_N$ be the number of distinct primitive chords carried in
$(\delta_m)$ for $m\in\{0,\ldots , N\}$.
According to Lemma~\ref{lem:linear}, there exist $a>0$ and $b\in\R$
such that $A_{aN+b} \geq A_N +1$.
Let $a'>a$, then for sufficiently large $N$,
$a'N > aN+b$ and $A_{a'N}\geq A_N +1$.
Thus, for all $k\geq 1$,
$A_{(a')^k N}\geq A_N + k$ and there exists
$c_0 >0$ such that
\begin{equation*}
    A_m \geq \frac{\log m}{\log a'} - c_0,\quad
    \forall m\in\N.
\end{equation*}

Paths in $(\delta_m)$ for $m\in\{ 0,\ldots, N\}$ have length
$\leq cN+d$ for some $c>0$ and $d\in\R$ according to Lemma~\ref{lem:lbound}
(and they are longer than the primitive chords they carry)
therefore,
\begin{equation*}
    2\CLd{\ell}
    \geq A_{\lfloor (\ell - d)/c\rfloor} \geq \frac{\log \ell}{\log a'} - c_1,
    \quad\forall \ell >0,
\end{equation*}
for some constant $c_1>0$
(in the case $p=q$, a primitive chord and its inverse
are geometrically identical, hence the factor of $2$).
\end{proof}

We go back to the general case where $M$ is not assumed to be closed.

\begin{lem}\label{lem:quadratic}
    There exists a quadratic polynomial $P\in\R[X]$ such that
    if $\delta_{m_1}$ and $\delta_{m_2}$ 
    carry the same primitive chord $d$ for some $m_1<m_2$,
    then for all $m>P(m_2)$ the chord
    $\delta_m$ does not carry $d$.
    Coefficients of $P$ can be made continuous in the base point $q\in M$.
\end{lem}

\begin{proof}
    Let $(\delta'_i) = (d\cdot c^{k_i})_i$ be the sub-sequence
    $(\delta_{m_i})$ of $(\delta_m)$ carrying the primitive
    chord $d$
    with $(m_i)$ being increasing
    and let $\kappa:=\max(k_1,k_2)$.
    Similarly to the proof of Lemma~\ref{lem:linear},
    the finite sequence $(k_i)$ is bounded by
    $n(\kappa +1)$ with $\kappa\leq \frac{am_2+b}{\length(c)}$,
    where $a,b\in\R$ are given by the upper-bound of
    Lemma~\ref{lem:lbound}.
    We fix any linear projection $ H_1(M,\Z)\to\la h\ra\simeq\Z$,
    $\beta\mapsto[\beta]$.
    By definition of $m_i$, 
    $[\alpha^{-1}\cdot d] +k_i[c] = m_i$.
    Let $u:=[\alpha^{-1}\cdot d]\in\Z$
    and $v:=[c]\in\Z$.
    Since $(k_2-k_1)v= m_2-m_1$ and $m_1 < m_2$, one has $|v|\leq m_2$
    and thus $|u|\leq m_2 + k_2 m_2$.
    Finally,
    \begin{multline*}
        m_i \leq |u| + k_i |v| \leq m_2 + \kappa m_2 + n(\kappa +1) m_2 \\
        \leq m_2 + \frac{am_2+b}{r} + n\left(\frac{am_2+b}{r}+1\right)m_2
        =: P(m_2),
    \end{multline*}
    where $r := \inf_{\gamma\in\Omega_{q},\ [\gamma]\neq 0}
    \length(\gamma) >0$ depends continuously in $q\in M$.
\end{proof}

\begin{proof}[Proof of Theorem~\ref{thm:loglog}]
    This is the same proof as for Theorem~\ref{thm:log}
    but now, for sufficiently large $N$, there exists some $a>1$ such that
    $A_{aN^2}\geq A_N +1$ with the same notations.
    Thus $A_{(aN)^{2^k}}\geq A_N + k$ and there exists $c_0>0$ such that
    \begin{equation*}
        \forall m\in\N,\quad A_m \geq
        \frac{\log(\log m)}{\log 2} - c_0.
    \end{equation*}
    Since the upper-bound given by Lemma~\ref{lem:lbound}
    is still true,
    we can conclude similarly.
\end{proof}

\bibliographystyle{amsplain}
\bibliography{biblio}

\end{document}